\documentclass[12pt, one side]{article}
\usepackage[english]{babel}
\usepackage[T1]{fontenc}

\usepackage[centertags]{amsmath}
\usepackage{amssymb}
\usepackage{amsthm}
\usepackage{newlfont}
\usepackage{amsfonts}
\theoremstyle{plain}

\newtheorem{theorem}{Theorem}[section]

\newtheorem{lemma}{Lemma}[section]

\begin{document}

\begin{center}
\textbf{\LARGE On The $b$-Chromatic Number \\\vspace{5pt}
 of  Regular Bounded Graphs }
\end{center}

 \begin{center}
{El Sahili Amine /  Kouider Mekkia }
\end{center}
\begin{center}{   Mortada Maidoun}
  \end{center}

 \indent        \centerline{\textbf{ Abstract}}\vspace{08pt}
 A $b$-coloring of a graph is a proper coloring such that every color
 class contains a vertex adjacent to at least one vertex in each of the
 other color classes.  The $b$-chromatic number of a graph $G$, denoted
by $b(G)$, is the maximum integer $k$ such that $G$ admits a
$b$-coloring with $k$ colors. El Sahili and Kouider conjectured that
$b(G)=d+1$ for $d$-regular graph with girth 5, $d\geq4$. In this
paper, we prove that this conjecture holds for $d$-regular graph
with at least $d^3+d$ vertices. More precisely we show that
$b(G)=d+$1 for $d$-regular graph with at least $d^3+d$ vertices and
containing no cycle of order 4. We also prove that $b(G)=d+1$ for
$d$-regular graphs with at least $2d^3+2d-2d^2$ vertices improving
 Cabello and Jakovac bound.
\\


\section{Introduction}
  A proper coloring of a graph $G = (V;E)$ is an assignment of colors to the
vertices of G, such that any two adjacent vertices have different
colors. The chromatic number of $G$, denoted by $\chi(G)$ is the
smallest integer $k$ such that $G$ has a proper coloring with $k$
colors. A color class in a proper coloring of a graph $G$ is the
subset of $V$ containing all the vertices of same color. A proper
coloring of a graph is called a $b$-coloring, if  each color class
contains a vertex  adjacent to at least one vertex of each of the
other color classes. Such a vertex  is called a dominant vertex. The
$b$-chromatic number of a graph $G$, denoted by $b(G)$,  is the
largest integer $k$ such that $G$ has a $b$-coloring with $k$
colors.
 For a given graph G, it may be easily remarked that
 $ \chi(G) \leq b(G)\leq \Delta(G) + 1$.\newline

 \indent The $b$-chromatic
number of a graph was introduced by Irving and Manlove $[10]$  when
considering minimal proper colorings with respect to a partial order
defined on the set of all partitions of the vertices of a graph.
They proved that determining  $b(G)$
 is NP-hard for general graphs, but polynomial-time solvable for
trees. \newline

Recently, Kratochvil \emph{et al}.$[11]$   have shown that
determining $b(G)$ is NP-hard even for bipartite graphs while
Corteel, Valencia-Pabon, and Vera $[5] $ proved that there is no
constant $\epsilon > 0$ for which the $b$-chromatic number can be
approximated within a factor of 120/133-$\epsilon$ in polynomial
time $(unless P=NP)$.\newline

Finally, Balakrishnan and Francis Raj $[1,2]$  investigated the
b-chromatic number of the Mycielskians and vertex deleted subgraphs.
Hoang and Kouider $[9]$ characterize all bipartite graphs $G$ and
all $P_4$-sparse graphs $G$ such that each induced subgraph $H$ of
$G$ satisfies $b(H) = \chi(H)$. In $[8]$, Effantin and Kheddouci
gave the exact value for the b-chromatic number of power graphs of a
path and  determined bounds for the b-chromatic number of
power graphs of a cycle.\\
\\
\text In $[6]$, Kouider and El-Sahili formulated the following
conjecture: For a $d$-regular graph with girth 5, $b(G)=d+1$ for
$d\geq 4$. They proved that for every graph $G$ with girth at least
$6$, $b(G)$ is at least the minimum degree of the graph, and if this
graph is $d$-regular then $b(G) = d + 1$.
 This conjecture have been proved in $[6]$ for
the particular case when $G$ contains no cycles of order 6. In
$[3]$, Maffray \emph{et.al} proved that the conjecture holds for
d-regular graphs different from the Petersen graph and $d \leq  6$.
The $f$-chromatic vertex number of a $d$-regular graph $G$, denoted
by $f(G)$, is the maximum number of dominant vertices with distinct
color classes in a proper $d + 1$-coloring of $G$. In $[7]$, El
Sahili \emph{et al}. proved that $f(G)\leq b(G)$ and then
reformulated El Sahili and Kouider conjecture as follows: $f(G)=d+1$
for $d$-regular graph with no cycle of order 4. They proved that (i)
$f(G)\geq \lfloor\frac{d+1}{2}\rfloor+2$ for a $d$-regular graph
containing no cycle of order 4; (ii) $f(G) \geq \lceil \frac{d+1}
{2}\rceil $ + 4 for $d$-regular graph  containing no cycle of order
4 and of diameter 5; (iii) $b(G) = d + 1$ for a $d$-regular graph
containing no cycle of order 4 nor of order 6; (v) $b(G) = d + 1$
for a $d$-regular graph with no cycle of order 4 and of diameter at
least 6. In this paper, we prove in two different methods that El
Sahili and Kouider conjecture holds for a $d$-regular graph
containing at least $d^3+d$ vertices and more precisely we prove
that $b(G)=d+1$ for a $d$-regular graph containing no cycle of order 4 and at least $d^3+d$ vertices.\\
\\
In $[11]$, Kratochvil \emph{et al.} proved  that for a $d$-regular
graph $G$ with at least $d^4$ vertices, $b(G) = d+1$. It follows
from their result that for any $d$, there is only a finite number of
$d$-regular graphs $G$ with $b(G) \leq d$. In $[4]$, using
matchings, Cabello and Jakovac reduced the bound of $d^4$ vertices
to $2d^3-d^2+d$ vertices. In this paper, we  show that, without
using  matching, $b(G)=d+$1 for a $d$-regular graphs with $v(G)\geq
2d^3+2d-2d^2$ improving Cabello and Jakovac bound.
\\
\\
\section{Lower Bounded Graphs }
 Consider a $d$-regular graph $G$ and let $K$ and
$F$ be 2 disjoint and fixed induced subgraphs of $G$. Suppose that
the vertices of $K$ are colored by a proper $d+1$-coloring. Also,
suppose that the vertices of $F$ are colored by a proper
$d+1$-coloring $c$. We define a   digraph $\Delta_c$ where
$V(\Delta_c)=\{1,2,...,d+1\}$ and $E(\Delta_c)=\{(i,j)\;:\;$a vertex
of color $i$ in $F$ is not adjacent to a vertex of color $j$ in
$K$\}. $\Delta_c$ is called a coloring digraph. Note that the
coloring digraph $\Delta_c$ may contain loops and circuits of length
2. The number of loops in $\Delta_c$ is denoted by $\ell(\Delta_c)$.
We introduce the following lemma:
\begin{lemma}
If $(i,i)\notin E(\Delta_c)$ and if there exists a circuit $C$ in
$\Delta_c$ containing $i$, then we can recolor $V(F)$ by a proper
$d+1$-coloring $c'$ such that $\ell(\Delta_{c'})> \ell(\Delta_c)$.
\end{lemma}
\begin{proof}
Suppose that $(i,i)\notin E(\Delta_c)$ and there exists a circuit
$C$ in $\Delta_c$ containing $i$. Without loss of generality,
suppose that $C=1\;2\;...\;i$. We define a new proper coloring $c'$,
where for $v\in F$ \begin{center}
 $c'(v) = \begin{cases}
  c(v) & \text{if $c(v)\notin C$} \\
  c(v)+1 & \text{if $c(v)\in C\backslash\{i\}$} \\
  1 & \text{if $c(v)=i$}
\end{cases}$
\end{center}
 A loop $(s,s)$ in $\Delta_c$ is clearly a loop in $\Delta_{c'}$ whenever $s\geq i+1$.   Since
 $(s,s+1)$  and $ (i,1)\in E(\Delta_c)$, $1\leq s\leq i-1$, then $(l,l)$  is a loop in
$\Delta_{c'}$, $\forall l \; 1\leq l\leq i$. Thus,
$\ell(\Delta_{c'})> \ell(\Delta_c)$.
\end{proof}

\begin{theorem}
Let $G$ be a $d$-regular graph with no cycle of order 4. If
$v(G)\geq d^3+d$, then $b(G)=d+1$
\end{theorem}
\begin{proof}
Suppose that $k$ vertices and their neighbors are colored by a
proper $d+1$-coloring in such a way that these $k$ vertices are
dominant of color 1,2,...,$k$, $k\leq d$. Let $C$ be the set of
colored vertices, then $|C|=k+kd\leq d(d+1)$.
\newline
Let $R_i=\{v\in R: v$ has exactly $i$ neighbors in $C$\}, $0\leq
i\leq d$, and set  $R=N(C)=\bigcup _{i=1}^{d}R_i$.\\ Let
$R_a=R_2\cup...\cup R_{\lfloor \frac{d+1}{2}\rfloor}$,
$R_b=R_4\cup...\cup R_{\lfloor \frac{d+1}{2}\rfloor}$, and
$R_c=R_{\lfloor \frac{d+1}{2}\rfloor+1}\cup...\cup R_d$.\\
  Since dominant vertices has
no neighbors in $V (G)\backslash C$ and  a neighbor of a dominant
vertex has at most $d-1$ neighbors in $R$, then  by double counting
the edges between $C$ and $(R_1\cup R_2\cup...\cup R_d)$ we can say
that:

\begin{center}$|R|+|R_a|+2|R_b|+\lfloor
\frac{d+1}{2}\rfloor|R_c|\leq d^2(d-1) ~~~~(a)$\end{center}Let
\begin{center}
 $S_1=\{v\notin C\cup R \;:\;|N(v)\cap R_a|\geq
d-2\}$\end{center}
\begin{center}$S_2=\{v\notin C\cup R\;:\;
|N(v)\cap R_b|\geq \lceil \frac{d-1}{2}\rceil\}$ \end{center}
\begin{center}
$S_3=\{v\notin C\cup R\;:\; |N(v)\cap R_c|\geq 1\}$
\end{center}

 We have \begin{center} $(d-2)|S_1|\leq (d-2)|R_a|,$ then
$|S_1|\leq|R_a|$
\end{center}
\begin{center}
$ \lceil \frac{d-1}{2}\rceil|S_2|\leq (d-4)|R_b|$, then $|S_2|<
2|R_b|$\end{center}

\begin{center}
 and $|S_3|\leq ( \lceil \frac{d-1}{2}\rceil-1)|R_c|$ \end{center} So, \begin{center}
$|C|+|R|+|S_1|+|S_2|+|S_3|< |C|+|R|+|R_a|+2|R_b|+( \lceil
\frac{d-1}{2}\rceil-1)|R_c|$\end{center} But, by $(a)$, we have  $
\lfloor \frac{d+1}{2}\rfloor |R_c|\leq d^2(d-1)-|R|-|R_a|-2|R_b|$,
thus
\begin{center} $|C|+|R|+|S_1|+|S_2|+|S_3|<
|C|+|R|+|R_a|+2|R_b|+d^2(d-1)-|R|-|R_a|-2|R_b|\leq d(d+1)+
d^2(d-1)\leq d^3+d$
 \end{center}
 Since $v(G)\geq d^3+d$, then there exists a vertex $y$ such that $y\notin C\cup
 R\cup S_1\cup S_2\cup S_3$. We note that
 \begin{center}
 $|N(y)\cap R_a|\leq
d-3$, \; $|N(y)\cap R_b|\leq \lceil \frac{d-1}{2}\rceil-1 ,\;
|N(y)\cap R_c|= 0$.
\end{center}

Thus, we have \begin{center} $|N(y)\cap(R_0\cup R_1)|\geq 3~~~~(b)$.
\end{center}
 Let $K$  and $F$ be two induced subgraphs of $G$ where
$V(K)=C$ and $V(F)=N(y)\cup\{y\}$. Color $y$ and its neighbors by a
proper $d$+1-coloring $c$ in such a way that $y$ is a dominant
vertex of color $k$+1 and $\ell(\Delta_c)$ is maximal . If
$\ell(\Delta_c)= d+1$, then $y$ is a dominant vertex in a proper
$d$+1-coloring for $V(K)\cup$ $V(F)$. Else, there exists $i\neq k+1$
such that $(i,i)\notin E(\Delta_c)$. Let $x$ be the vertex of color
$i$ in $N(y)$.

   There exists at least a neighbor of $y$,
say $w_1$, such that $w_1$ has no neighbor of color  $i$ in $K$
since $G$ has no $C_4$ and $k\leq d$. Also, by $(b)$, there exist at
least three neighbors of $y$, say $w_2, w_3$ and $w_4$,  such that
each of them has at most one neighbor  in $K$. Since $x$ has  at
most $\lfloor\frac{d+1}{2}\rfloor$  neighbors  in $K$ and $y\notin
C\cup R\cup S_1\cup S_2\cup S_3$, then there exists a neighbor of
$y$, say $w_5$, such that $x$ has no neighbor of color $c(w_5)$
$(c(w_5)\neq k+1)$ in $K$ and $w_5$ has at most 3 neighbors in $K$.
If $x$ has no neighbor of color $c(w_1)$ in
 $K$, then $C_1=i\;c(w_1)$ is a circuit in $\Delta_c$. Else, $i$ has
 a neighbor in $K$ of color $c(w_1)$. If there exists $j$, $2\leq j\leq 4$, such
 that $x$ has no neighbor of color $c(w_j)$ in $K$, then $C_2$ or $C_3$ is a circuit in
$\Delta_{c}$, where $C_2=i\;c(w_j)$ and $C_3=i\;c(w_j)\;c(w_1)$,
since $w_j$ has at most one  neighbor in $K$.  Otherwise, neither
$x$ nor  $w_5$ belongs to the set $\{w_1,w_2,w_3,w_4\}$ since $x$
has no neighbor in $K$ of color $c(w_5)$ and it has more than one
neighbor. If $w_5$ is not adjacent to a vertex of color
 $i$ in $K$, then $C_3=i\;c(w_5)$ is a circuit in $\Delta_{c}$. Otherwise, there exists $j,\; 1\leq j\leq 4$,
such that $w_5$ has no neighbor of color $c(w_j)$ in $K$ since $w_5$
has at most 3 neighbors in $K$. If j=1, then $C_4=c(w_5)\;c(w_1)\;i$
is a circuit in $\Delta_c$. Else, $C_5$ or $C_6$ is a circuit in
$\Delta_{c}$, where $C_5=c(w_5)\;c(w_j)\;i$ and
$C_6=c(w_5)\;c(w_j)\;c(w_1)\;i$, since $w_j$ has at most one
neighbor in $K$. In all cases, there exists a circuit containing
$i$. Then, by Lemma 2.1 we can find a proper $d+$1-coloring $c'$ of
$V(F)$ such that $\ell(\Delta_{c'})>\ell(\Delta_c)$, a
contradiction. Thus, $\ell(\Delta_c)=d+1$ and so $c$ is a proper
$d$+1-coloring for $V(F) \cup V(K)$ and $y$ is a dominant vertex of
color $k+1$. This proves that we can find a proper $d + 1$-coloring
of $G$ that contains $d + 1$ dominant vertices of distinct colors
\end{proof}
 \text \indent The coloring digraph, which is used to prove Theorem 2.1, can be used  also to establish
 the following result improving Cabello and Jakovac bound:

\begin{theorem}
Let $G$ be a $d$-regular graph such that $v(G)\geq2d^3+2d-2d^2$,
then $b(G)=d+1$.
\end{theorem}
\begin{proof}
Suppose that $k$ vertices and their neighbors are colored by a
proper $d+1$-coloring in such a way that these $k$ vertices are
dominant of color 1,2,...,$k$, $k\leq d$. Define $C$, $R$ and $R_i$,
$0\leq i\leq d$, as in the  proof of Theorem 2.1. Let $C_i=\{v\in R:
\; v$ has a neighbor of color $i$ in $C\}$, $1\leq i\leq d+1$. Let
$R_a=R_3\cup...\cup R_{\lfloor\frac{d+1}{2}\rfloor}$ and $
R_b=R_{\lfloor\frac{d+1}{2}\rfloor+1}\cup...\cup R_d$.\newline
Dominant vertices has no neighbors in $V (G)\backslash C$ and a
neighbor of a dominant vertex has at most $d-1$ neighbors in $R$, so
we can say that:
\begin{center} $|R|+2|R_a|+\lfloor\frac{d+1}{2}\rfloor|R_b| \leq d^2(d-1) $
\end{center} For $k<d$, we have \begin{center}
 $|C_i|\leq  k(d-1)\leq (d-1)^2\;\;\; 1\leq i\leq d+1$
 \end{center}
 For $k=d$, since only $d-1$ vertices of color $i,\; i\neq d+1$,
 can have neighbors outside $C$ while $d$ vertices of color $d+1$ can have neighbors outside $C$, then
 \begin{center}
 $|C_i|\leq  (d-1)^2 $ for$ \; i\neq d+1$, and $|C_{d+1}|\leq d(d-1)$
 \end{center}
 Let \begin{center}

 $S_1=\{v\notin C\cup R \;:\;|N(v)\cap R_a|\geq\lceil\frac
{d-1}{2}\rceil\}$\end{center}
\begin{center}
$S_2=\{v\notin C\cup R\;:\; |N(v)\cap R_b|\geq 1\}$
\end{center}
\begin{center}
$S'_i=\{v\notin C\cup R\;:\; |N(v)\cap C_i|\geq d-1\}$, $1\leq i\leq
d+1$ such that $i\neq k+1$.
\end{center}
The union of the sets $S'_i$, $1\leq i\leq d+1$ and $i\neq k+1$, is
denoted by $S'$. We have
\begin{center} $\lceil\frac {d-1}{2}\rceil|S_1|\leq (d-3)|R_a|,$ thus
$|S_1|< 2|R_a|$
\end{center}

\begin{center}
  $|S_2|\leq (\lceil\frac
{d-1}{2}\rceil-1)|R_b|$ \end{center}\begin{center}
  $(d-1)|S'_i|\leq (d-1)|C_i|$, thus $|S'_i|\leq (d-1)^2$ \end{center}
  \begin{center} and $|S'|\leq d(d-1)^2$.\end{center} So, \begin{center}
$|C|+|R|+|S_1|+|S_2|+|S'|< |C|+|R|+2|R_a|+(\lceil\frac
{d-1}{2}\rceil-1)|R_b|+ d(d-1)^2$\end{center} But
$\lfloor\frac{d+1}{2}\rfloor|R_b|\leq d^2(d-1)-|R|-2|R_a|$, thus
\begin{center} $|C|+|R|+|S_1|+|S_2|+|S'|<
 d(d+1)+ d^2(d-1)+d(d-1)^2\leq
2d^3+2d-2d^2$
 \end{center}
 Since $v(G)\geq 2d^3+2d-2d^2$, then there exists a vertex $y$ such that $y\notin C\cup
 R\cup S_1\cup S_2\cup S'$. We note that
 \begin{center}
 $|N(y)\cap R_a|\leq\lceil\frac
{d-1}{2}\rceil-1$,  $|N(y)\cap R_b|=0$, $|N(y)\cap C_i|\leq d-2$,
$\forall i\neq k+1~~~~(*)$. \end{center} Let $K$  and $F$ be two
induced subgraphs where $V(K)=C$ and $V(F)=N(y)\cup\{y\}$. Color $y$
and its neighbors by a proper $d$+1-coloring $c$ in such a way that
$y$ is a dominant vertex of color $k$+1 and $\ell(\Delta_c)$ is
maximal . If $\ell(\Delta_c)= d+1$, then $y$ is a dominant vertex in
a proper $d$+1-coloring for $V(K)\cup V(F)$. Else, there exists
$i\neq k+1$, such that $(i,i)\notin E(\Delta_c)$. Let $x$ be the
vertex of color $i$ in $N(y)$.
\newline By $(*)$, we can find at least 2 neighbors of $y$,
say $w_1$ and $w_2$, such that $w_1$ and $w_2$ has no neighbor of
color $i$ in $K$. If $x$ has no neighbor of color $c(w_j)$,
$j\in\{1,2\}$, then $C_1=i\;c(w_j)$ is a circuit in $\Delta_c$.
Otherwise, since $x$ has  at most $\lfloor\frac{d+1}{2}\rfloor$
neighbors  in $K$, then there exists a neighbor
 of $y$, say $w_3$,  such that $x$ has no neighbor of color $c(w_3)$ in $K$ where $c(w_3)\neq k+1$, $w_3\notin \{w_1,w_2\}$ and $w_3$ has at most 2
 neighbors in $K$.  If $w_3$ has no neighbor of color $i$
in $K$, then $C_2=c(w_3)\;i$ is a circuit in $\Delta_c$. Else, $w_3$
has no neighbor in $K$ of color $c(w_k)$ where $k=1$ or $2$. So,
$C_3=c(w_3)\;c(w_k)\;i$ is a circuit in $\Delta_{\zeta}$. In all
cases, there exists a circuit containing $i$, then by Lemma 2.1 we
can find a proper $d+$1-coloring $c'$ of $V(F)$ such that
$\ell(\Delta_{c'})>\ell(\Delta_{c})$, a contradiction. Thus,
$\ell(\Delta_{c})=d+1$ and so $c$ is a proper $d$+1-coloring for
$V(F) \cup V(K)$ and $y$ is a dominant vertex of color $k+1$.
Consequently, we can find a $d+1$ dominant vertices of distinct
colors.
\end{proof}

\section {Matching and $b$-coloring}
  Using matching Cabello and Jacovac proved that $b(G)=d+1$ for any $d$ -regular graph with at least $2d^3 +d-d^2$ vertices.
  Matching also yields another proof for Theorem 2.1. This proof is based on the following
  Lemma:
\begin{lemma}
Let $t$ be a fixed integer. Let $L$ and $V$ be two sets of
cardinality $t$. Let $H$ be a bipartite graph with partition $V$ and
$L$ such that for every $v \in V$ and every $u \in L$,
 $d_H (v)+d_H (u)\geq t$.
Then $H$ has a perfect matching.
\end{lemma}
\vspace{3mm}

\begin{proof} The proof is by contradiction. Let $M$ be the maximum matching such that $M$ is not perfect. Then, there exist at least
two vertices, say  $u\in L$ and $v\in V$, outside $M$. Since $d_H
(v)+d_H (u)\geq t$, then there exists an edge in $M$, say $ab$, such
that $a \in N_H(u)$ and $b\in N_H(v)$. Then let $M'$ be the set of
edges such that $M'=( M\backslash\{ab\})\cup\{au, bv\}$. It is clear
that $M'$ is a matching with $|M'|>|M|$, a contradiction.\end{proof}
\emph{ Another
Proof of Theorem 2.1}.\\
 We have a partial $b$ coloring of the graph
$G$, the set of $k$ dominant vertices of the colors $1,2,..,k$ and
their neighbors. $C$, $R$ and $R_i$, $0\leq i\leq d$, are defined as
in the previous proof. Let $C_i=\{v\in C: \; v$
is of color $i$  $\}$, $1\leq i\leq d+1$.\\
Let $R_a=R_1 \cup R_2 \cup ...\cup
R_{\lfloor\frac{d+1}{2}\rfloor}$,\;$R_b=R_3 \cup R_4 \cup
 ...\cup R_{\lfloor\frac{d+1}{2}\rfloor}$, and $R_c=R_{\lfloor\frac{d+1}{2}\rfloor+1}\cup...\cup R_d$. \\Dominant vertices has no neighbors in $V (G)\backslash C$
  and a neighbor of a dominant vertex has at most d - 1 neighbors in $R$, so we can say that:
  \begin{center}$|R| +|R_2| + 2|R_b|+ {\lfloor\frac{d+1}{2}\rfloor}|R_c| \leq d^2(d - 1)
  ~~~~(a)$\end{center}
For each $y\in G-(C\cup R)$, let us set:\\
$R_c(v)= N(v)\cap R_c$,\\
 $R_b(v)=N(v)\cap R_b$,\\
 $R'_b(v)=N(v)\cap (R_0 \cup R_1\cup R_2)$,\\
 $S_1= \{ v \notin (C\cup R),|R_b(v)| \geq \lceil \frac{d-1}{2}\rceil \}$.\\
$S_2 = \{ v \notin(C\cup R),|R_c(v)|\geq 1 \}$.\\
Then, we have
\begin{center}
$\lceil \frac{d-1}{2}\rceil |S_1|\leq (d-3) |R_b|,$ then $|S_1|\leq
2|R_b|$
\end{center}
\begin{center}
and $|S_2|\leq (\lceil \frac{d-1}{2}\rceil-1)|R_c|$
\end{center}

Let $S_0=\{y \in G\backslash(C\cup R \cup S_1\cup S_2):\; N(v)\cap
(R_2\cup...\cup R_{\lfloor\frac{d+1}{2}\rfloor+1})\geq d-2\}$  and
$S_{0,i}=\{v\in S_0: R_b(v)=i\} $, $0 \leq i \leq \lceil
\frac{d-1}{2}\rceil-1$. So, $S_0= \cup_{0\leq i \leq \lceil
\frac{d-1}{2}\rceil-1} S_{0,i}.$ \\We have: \begin{center} $
\sum_{0\leq i \leq \lceil \frac{d-1}{2}\rceil-1} i.|S_{0,i}|+\lceil
\frac{d-1}{2}\rceil|S_1| \leq (d-3).|R_b| ~~~~(1)$ \end{center}
\begin{center}
$\sum_{0 \leq i \leq (\lceil \frac{d-1}{2}\rceil-1)}(d-2-
i).|S_{0,i}|)\leq (d-2).|R_2|  ~~~~(2)$ \end{center} From these 2
last inequalities, we deduce that\
$$(d-2).|S_0| +2\lceil \frac{d-1}{2}\rceil|S_1|\leq 2(d-3)|R_b|+ (d-2)|R_2|$$
Thus, we get \begin{center} $|S_0| +|S_1|< 2.|R_b|+|R_2|~~~~(b)$
\end{center}

Thus, by $(a)$ and $(b)$, we get: \begin{center}
$|C|+|R|+|S_0|+|S_1|+|S_2|<|C|+|R|+2|R_b|+|R_2|+(\lceil
\frac{d-1}{2}\rceil-1)|R_c|\leq d^2(d-1)+d(d+1)\leq d^3+d$
.\end{center} Since, $|V(G)|\geq(d^3+d)$, then we can find a vertex
$y$ such that $y \in G-(C\cup R \cup S_0\cup S_1\cup S_2)$. Color
$y$ by $k+1$. Now, we color separately $R_b(y)$ and $R'_b(y)$. Let
$B$ be a subset of $N(y)$. For any color $j$, let $e(C_j,B)$ be the
number of edges with one extremity in $C_j$ and the other one in
$B$. We remark that for any color $j$\begin{center} $ e(C_j ,N(y))
\leq d - 1 ~~~~(*)$
\end{center}

By definition of $y$,
 $|R_c(y)|=0,\; |R_b(y)|\leq \lceil \frac{d-1}{2}\rceil-1  $ and
$|N(y)\cap (R-R_1 )|\leq (d-3)$, so \begin{center}$|N(y)\cap (R_0
\cup R_1 )| \geq 3 ~~~~(**)$\end{center} Let us note that \
$$e(C,R'_b(y))=|N(y)\cap R_1|+ 2.|N(y)\cap R_2|$$
Let $1$ be a color such that $e(C_1,R'_b(y))$ is maximum.\\

\rm (1) If $e(C_1,R'_b(y))=|R'_b(y) \setminus  R_0|$, by (*) we can
choose one vertex $w_1 \notin R'_b(y)$ not neighbor of $c_1$ and we
color it by $c_1$.

 We have:
 $e(C\backslash C_1,R'_b(y))=|R'_b(y)\cap R_2 | \leq |R'_b(y)|-3$ by inequality (**).
\vspace{2mm}

(2) If $e(C_1,R'_b(y))=|R'_b(y) \backslash R_0|-1$, we choose one
vertex $w_1 \in R'_b(y)$ not neighbor of $c_1$ and we color it by
$c_1$. \vspace{2mm}
 Now, by (**),
$e(C-C_1,R'_b(y)-{w_1})=|R'_b(y)\cap R_2| \leq
|R'_b(y)\backslash\{w_1\}|-2$ 

(3) If (1) and (2) are excluded, for any color j, $e(C_j,R'_b(y))
\leq |R'_b(y) \setminus R_0|-2.$

\rm \vspace{2mm}
 Now, we are going to color $R_b(y)\cup R'_b(y)\setminus\{w_1\}$ using colors in $L$, where $L=\{1,2,..,d+1 \}- \{1,(k+1)\}
 \}$.
By definition of $R_b(y)$, each vertex $u$  of $R_b(y)$ is adjacent
to at most  $ \lceil \frac{d-1}{2}\rceil-1$ colored vertices in $C$
, so $u$ is colorable by at least $\lfloor \frac{d+1}{2}\rfloor$
colors of $L$ in a proper coloring. Thus, we can color easily  the
vertices of $R_b(y)$ by colors of $L$ such that 2 by 2 they get
different colors. \vspace{2mm}

There remains a set $L' \subset \{1,..,d \}$ (or $\{2,...,d\}$ ) of
$|R'_b(y) \setminus w_1\}|$  colors not used yet.

\rm For each remaining color $j$, we have\
 $e(C_j,R'_b(y)\backslash\{w_1\})\leq |R'_b(y)\setminus \{w_1\}|-2$, and
for each $u \in R'_b(y)$, $u$ has at most two colored neighbors in
$C$. Let $H$ be the bipartite graph with bipartition $L'$ and
$V=R'_b(y) \setminus \{w_1\}$ such that $uj$ is an edge in $H$
whenever $u$ has no neighbor of color $j$ in $G$, where $u\in V$ and
$j\in L'$. Let $ t= |R'_b(y)\setminus\{w_1\}|$. For each $u \in V$,
$d_H(u) \geq t-2$, for each $j\in L'$, $d_H(j)\geq 2$. Thus, by
lemma 3.1, there exists a perfect matching in $H$. Now, if $uj$ is
an edge in the matching then color $u$ by $j$. Finally, we get a
dominant vertex
$y$ for the color $k+1$. $\Box$\\

\end{document}